\documentclass[11pt,lefteqn]{article}

\usepackage{amsmath,amsfonts,amsthm,amssymb,graphics,color}

\newcommand{\C}{\mathbb C}

\newcommand{\D}{\mathbb D}

\renewcommand{\Re}{\mathop {\rm Re}\nolimits}

\newtheorem{theorem}{Theorem}
\newtheorem{lemma}[theorem]{Lemma}
\newtheorem{corollary}[theorem]{Corollary}

\newtheorem{remark}[theorem]{Remark}

\begin{document}
\date {\today}
\title{The numerical range as a spectral set}

\author{Michel Crouzeix\footnote{Universit\'e de Rennes, email: Michel.Crouzeix@univ-rennes1.fr}  \hspace{2pt}  and C\'esar Palencia\footnote{Universidad de Valladolid, email: cesar.palencia@tel.uva.es}}

\maketitle

\begin{abstract} 
It is shown that the numerical range of a linear operator operator in a Hilbert space  is a (complete) $(1{+}\sqrt2)$-spectral set. The  proof relies, among other things,  in the behavior of the Cauchy transform of the conjugates of  holomorphic functions.
\end{abstract}

\paragraph{2000 Mathematical subject 3assifications\,:}47A25 ; 47A30

\noindent{\bf Keywords\,:}{ numerical range, spectral set}

\section{Introduction}
Let us consider a smooth, bounded, convex domain  $\Omega \subset\C$. In a seminal paper\,\cite{dede}, Bernard and Fran\c{c}ois Delyon showed that there exists a best constant $C_{\Omega}$ such that, for all rational functions $f$, there holds
\begin{equation}\label{bfd}
\|f(A)\|\leq C_{\Omega} \sup_{z\in\Omega}|f(z)|,
\end{equation}
whenever $A$ is a bounded linear operator in a complex Hilbert space $(H, \langle, \rangle, \| \,\|)$ whose numerical range 
$$
W(A): = \{ \, \langle Av,v\rangle \, : \,  v\in H,  \, \|v\|=1 \, \}
$$
satisfies $\overline{W(A)}\subset\Omega$.

Their work has inspired the conjecture\,\cite{crzx} $\mathcal{Q}:=\sup_{\Omega} C_{\Omega}=2$  and it has been shown in \,\cite{crac} that $2\leq \mathcal{Q}\leq 11.08$.  Although there is numerical support to it \cite{grlo},  the conjecture $\mathcal{Q} = 2$ remains to be an open problem and we refer to \cite{choi,chgr, grlo,crzx1} and the bibliography in \cite{crzx1} for the relevant background on the issue.

The aim of this paper is to present the improvement  (Theorem~\ref{improved} below) 
\begin{equation}
\label{our1}
2\leq \mathcal{Q}\leq 1{+}\sqrt2.
\end{equation}
Note that,  due to Mergelyan theorem, estimate \eqref{bfd} is valid, not only for rational functions $f$, but also for any $f$ belonging to the algebra
$$
\mathcal{A}(\Omega):=\{ \, f\, : \,  f  \mbox{ is holomorphic in } \Omega \mbox{ and continuous in }\overline\Omega \, \}.
$$
Furthermore, by using a sequence of smooth convex domains
 $\Omega_n\supset \overline{W(A)}$ converging to $\overline{W(A)}$, from (\ref{our1}) we easily get
 \begin{equation}
\label{our11}
 \|f(A)\|\leq (1{+}\sqrt2)\sup_{z\in W(A)}|f(z)|,
 \end{equation}
which shows that the numerical range $W(A)$ is a $(1{+}\sqrt2)$-spectral set for the operator $A$.
Furthermore, since $C_\Omega$ is uniformly bounded, \eqref{bfd} is still valid for all convex domains, even for unbounded ones, which allows to extend (\ref{our11}) to unbounded operators under classical suitable conditions.

For the sake of simplicity, we work with complex-valued functions, but there is no difficulty in generalizing the proof we give of (\ref{our1})  to matrix-valued mappings $f$, without changing the constant. Therefore\,\cite{paul}, the homomorphism $f\mapsto f(A)$, from the algebra ${\mathcal A}(W(A))$ into $B(H)$, is completely bounded by $1{+}\sqrt2$. In other words, the numerical range $W(A)$ is a complete $(1{+}\sqrt2)$-spectral set for the operator $A$.

Let us recall that the numerical radius $w(B)$ of a linear operator $B$ in the Hilbert space $H$ is the number
$$
w(B)=\sup_{z \in W(B)} |z|.
$$
Then, after (\ref{our1}), the interesting result \cite[Theorem~3.1]{dpw} implies
\begin{equation}
\label{others}
w(f(A)) \le \sqrt{2}\,  \sup_{z\in W(A)}|f(z)|, 
\end{equation}
for all rational functions bounded in $W(A)$, an estimate which also holds in the complete version. In the terminology of \cite{dpw},  this means that $W(A)$ is a complete $\sqrt2$-radius set for the operator $A$.

Let us point out that our approach to (\ref{our1})  is based on the Cauchy transform and only uses elementary tools. In particular, we do not use  dilation theory, which has shown its efficiency in the case where $\Omega$ is a disk.

The paper is organized in two sections. Section~\ref{AL} is devoted to some auxiliary lemmata, one of them (Lemma~\ref{average}) studies the behavior of the Cauchy transform $g$ of the conjugate of $f \in \mathcal{A}(\Omega)$ up to the boundary of $\Omega$, an interesting issue that is addressed in the maximum norm setting by using the double layer potential. This, combined with a representation for the balance $f(A) + g(A)^*$ (Lemma~\ref{lemaS}) are the tools for the proof of the main result (\ref{our1}),   presented in Section~\ref{MR}. The proof of (\ref{our1}) also shows that
$$
\| f (A ) \| \le 2\, \| f\|_\infty,
$$
in case $f$ takes values in some sector with vertex at the origin and angle $\pi/2$, as commented in  final Remark~\ref{final}. 

\section{Auxiliary lemmata}
\label{AL}
The boundary $\partial \Omega$ of the open, bounded, convex set $\Omega \subset \C$ is assumed to be smooth.  In the following, the algebra $\mathcal{A}(\Omega)$ is provided with the norm
$$
\|f\|_{\infty}=\max \{ \, |f(z)| :   z\in\overline\Omega \, \}, \qquad f \in \mathcal{A}(\Omega).
$$
Besides, $\mathcal{C}(\partial \Omega)$ stands for the set of the complex continuous functions on $\partial \Omega$, endowed with the norm
$$
 \| \varphi \|_{\partial \Omega} := \max \{ \, |\varphi (\sigma )|:\sigma \in \partial \Omega \, \}, \qquad \varphi \in \mathcal{C}(\partial \Omega).
$$

For  $\sigma \in \partial   \Omega $, the corresponding unit outward normal vector is denoted by $\nu=\nu(\sigma)$ and, for $\sigma \in \partial \Omega$ and $z \in\C\setminus\{ \sigma\}$, we introduce the double layer potential
$$
\mu(\sigma,z) =\frac{1}{2\pi }\Big(\frac{\nu}{\sigma -z }+\frac{\overline{\nu}}{\overline{\sigma} -\bar z }\Big),
$$
where $\nu = \nu(\sigma)$. It is geometrically clear that the set
$$
\Pi _\sigma :=\{ \, z\in \C : \Re(\nu (\bar\sigma{-}\bar z))>0 \, \}=\{ \, z\in \C\setminus\{\sigma\}: \mu (\sigma ,z)>0 \, \}
$$ 
is the open half-plane tangent containing $\Omega$ which is tangent  to $\partial \Omega$ at point  $\sigma$. Therefore, since $\Omega$ is convex, the claim $z\in \Omega$ is equivalent to say that $\mu (\sigma ,z)>0$, for all $\sigma \in\partial \Omega$. Analogously, we also note that $\mu(\sigma,\sigma_0) \ge 0$, for $\sigma, \, \sigma_0 \in \partial \Omega$ such that $\sigma \ne \sigma_0$.

Furthermore, we use a counterclockwise oriented, arclength parametrization $\sigma(s)$ of  $\partial \Omega$. Then, $\sigma(\cdot) $ is $L$ periodic, with $L$ the length of $\partial \Omega$. It is  noteworthy that $\nu(\sigma(s))=\sigma'(s)/i$ and
\begin{equation}
\label{muarg}
\mu (\sigma(s) ,z)=\frac{1}{\pi}\frac{d\arg(\sigma(s){ -}z)}{ds }, \qquad \forall z \ne \sigma(s),
\end{equation}
where $\arg$ stands for any continuous branch of the argument function defined in some neighborhood of $\sigma(s)-z \ne 0$.
In the light of this identity, it is also geometrically clear that
\begin{equation}
\label{intmu}
\int_{\partial \Omega} \mu(\sigma,z) \, ds = 2, \quad \forall z \in \Omega,  \quad \mbox{  and  } \quad  \int_{\partial \Omega} \mu(\sigma,\sigma_0) \, ds = 1,  \quad \forall \sigma_0 \in \partial \Omega ,
\end{equation}
and in particular $\mu(\sigma(\cdot), \sigma_0)$ is a density of probability for $\sigma_0 \in \partial\Omega$.

We define the Cauchy transform of a complex function $\varphi $, defined at least on the boundary of $\Omega$ and continuous on it, as
$$
C(\varphi ,z)=\frac{1}{2\pi  i}\int_{\partial \Omega}\varphi (\sigma )\frac{d\sigma}{\sigma -z} ,\quad \text{for } z\in\Omega.
$$
Whereas $C(\varphi,\cdot)$ is holomorphic in $\Omega$, the behavior of $C(\varphi,z)$ as $z \in \Omega$ approaches a boundary point, in general, is not clear. In Lemma~\ref{average} below,  we address this issue when $\varphi$ is the boundary value of the conjugate of $f \in \mathcal{A}(\Omega)$, which is  the situation of interest in the present paper.  
\begin{lemma}
\label{average}
 Assume that $f\in \mathcal A(\Omega)$. Then  $g = C(\bar f,\cdot)$ belongs to ${\cal A}(\Omega)$  and satisfies
 $$
\| g \|_\infty \le \|f \|_\infty.
$$
Furthermore,  $g(\partial \Omega) = \{ \, g(\sigma) \, : \, \sigma \in \partial \Omega \, \}$ is contained in the convex hull $\mbox{\rm conv} (\overline{f(\partial \Omega)}) $ of $\overline{f(\partial \Omega)}= \{ \, \overline{f(\sigma)} \, : \, \sigma \in \partial \Omega \, \}$. 
\end{lemma}
\begin{proof} 
Clearly, $g$ is holomorphic in $\Omega$. Besides,
conjugating in the Cauchy formula
$$
f(z) = \frac{1}{2 \pi i} \int_{\partial \Omega} \,  \frac{f(\sigma) }{\sigma-z}\, d\sigma =
\frac{1}{2 \pi}  \int_{\partial \Omega} \, \frac{f(\sigma) \nu(\sigma)}{\sigma-z} \, ds, \qquad z \in \Omega,
$$
leads to
\[
\label{cauchy4g}
g(z) = \int_{\partial \Omega}\overline{f(\sigma )}\,\mu(\sigma,z)\,ds-\overline{f(z)} .
\]
Along the well-known jump formula (discovered by C.\,Gauss around 1815),
if $z$ tends to $\sigma _0\in \partial \Omega$,  then $g(z)$ tends to $g(\sigma _0)$
defined on the boundary by
\begin{equation}
g(\sigma _{0})=\int_{\partial \Omega\setminus\{\sigma _0\}}\overline{f(\sigma )}\,\mu(\sigma,\sigma _0)\,ds, \quad\text{for  } \sigma_0\in \partial  \Omega.
\end{equation}
It is known since Carl Neumann  \cite{neu} that with this extension the function $g$ is continuous in $\overline{\Omega}$,
see for instance \cite[Theorem 3.22]{fo}. (As noticed by Neumann, the radial continuity
up to the boundary follows easily while the global continuity requires a careful analysis).

Since $\mu (\sigma ,\sigma _0) \, ds$ is a probability measure on $\partial \Omega$, it follows 
from \eqref{cauchy4g} that $g(\partial \Omega)$ is contained in the closed convex hull of $\overline{f(\partial \Omega)}$ which, by  compactness,  coincides with $\mbox{\rm conv} (\overline{f}(\partial \Omega)) $. Using the maximum principle, we get $\| g \|_\infty \le \|f \|_\infty$.

\end{proof}
\begin{remark}
The first part of this lemma is still valid if $\Omega$ unbounded but, generally in this case, $\mu (\sigma ,\sigma _0) \, ds$ is no longer a probability measure, so that the convex hull part is not guarantee.
\end{remark}

The rest of the section concerns  the Hilbert space setting. Given a complex Hilbert space  $(H, \langle \cdot, \cdot \rangle)$, both the norm in $H$ and the induced norm in the algebra $B(H)$ of bounded linear operators on $H$ are denoted by $\| \cdot \|$.

For $A \in B(H)$ and $\sigma \in \partial \Omega$ in the resolvent of $A$,  we set
$$
\mu (\sigma ,A):=\frac{1}{2\pi}\big(\nu(\sigma {-}A)^{-1}+\overline{\nu}(\overline{\sigma} {-}A^*)^{-1}\big).
$$
It turns out that
\begin{equation}
\label{wgeo}
\overline{W(A)}\subset\Omega \quad \Longrightarrow \quad\nu(\overline\sigma {-}A^*)+\overline{\nu}(\sigma {-}A)>0
 \quad \Longrightarrow \quad \mu (\sigma ,A) >0,
\quad \forall \sigma \in\partial \Omega.
\end{equation}

Under the assumption  $\overline{W(A)} \subset \Omega$, it is also meaningful to define
$$
C(\varphi ,A)=\frac{1}{2\pi  i}\int_{\partial \Omega}\varphi (\sigma )(\sigma {-} A)^{-1} d \sigma  \in B(H).
$$
Notice that $C(\varphi,A)$ does not corresponds to $\varphi (A)$, unless $\varphi$ can be extended to a  member of $\mathcal{A}(\Omega)$.

\begin{lemma}
\label{lemaS}
For $\varphi \in \mathcal{C}(\Omega)$, let us set
$$
S(\varphi,A)=C(\varphi,A)+C( \overline{\varphi},A)^* \in B(H).
$$
Then we have
$$
\|S(\varphi,A)\|\leq 2 \| \varphi \|_{\partial \Omega} .
$$
\end{lemma}

\begin{proof}
Let $\varphi : \partial \Omega \mapsto \C$ be continuous and let us assume, without loss of generality, that $\| \varphi \|_{\partial \Omega}=1$. It follows from the definition of $C(\varphi ,A) $ that
\begin{equation}
\label{S=CC}
S(\varphi,A)= \int_{\partial \Omega}\varphi (\sigma )\mu (\sigma ,A) \, ds.
\end{equation}

The remaining part of this proof is classical, we include it by the convenience of the reader. We first observe that, by Cauchy formula,  there holds  
$$
\int_{\partial \Omega} \mu(\sigma,A) \,ds=2.
$$ 
Besides,  for $\sigma \in \partial \Omega$, the operator $\mu (\sigma ,A)$ is self-adjoint and positive (\ref{wgeo}). Therefore,  for   $x, \, y \in H$, we have
 \begin{align*}
 |\langle S(\varphi,A)x,y\rangle|&=\Big|\int_{\partial \Omega}\varphi (\sigma )\langle\mu (\sigma ,A)x,y\rangle\,ds\Big|
 \leq \int_{\partial \Omega}|\langle\mu (\sigma ,A)x,y\rangle|\,ds\\
 &\leq  \int_{\partial \Omega}\langle\mu (\sigma ,A)x,x\rangle^{1/2}\langle\mu (\sigma ,A)y,y\rangle^{1/2}\,ds\\
 &\leq  \Big(\int_{\partial \Omega}\langle\mu (\sigma ,A)x,x\rangle\,ds\Big)^{1/2}
  \Big(\int_{\partial \Omega}\langle\mu (\sigma ,A)y,y\rangle\,ds\Big)^{1/2}\\
 & =  \Big\langle\int_{\partial \Omega}\mu (\sigma ,A)\,ds\  x,x\Big\rangle^{1/2}
 \Big\langle\int_{\partial \Omega}\mu (\sigma ,A)\,ds\ y,y\Big\rangle^{1/2} \\
&=2\,\|x\|\,\|y\|,
 \end{align*}
so that
$$
\| S(f,A) \| = \sup_{\|x\|=1, \|y\|=1}   |\langle S(\varphi,A)x,y\rangle| \le 2.
$$
\end{proof}

\begin{remark}
 In the particular case of the unit disk  $\Omega=\D$ and for $f\in \mathcal{A}(\Omega)$, it follows from Cauchy formula
 that $C( \overline{f},z)=\overline{f(0)}$ and $C( \overline{f},A)^* =f(0)$. Then, after the representation $f(A)=S(f,A)-f(0)$, a direct application of Lemma~\ref{lemaS} yields the famous Berger-Stampfli estimate \cite{bs} 
 \[
 \|f(A)\|\leq 2\,\|f\|_{\infty},\qquad \text{if\quad}f(0)=0.
 \]
Let us point out that a further result of Okubo and Ando \cite{okan} shows that this estimate remains valid even  if $f(0)\neq 0$.
Note that the proof of Berger and Stampfli, as the ones of Okubo and Ando, are based on  dilation theory. 
 \end{remark}
\begin{remark}\label{r6}
If $\Omega$ is unbounded, there exists a greater $\alpha\geq 0$ such that $\Omega$ contains a sector of angle $2\alpha$. Then this lemma is still valid with the improvement
 $$
\|S(\varphi,A)\|\leq 2\,\frac{\pi {-}\alpha}{\pi } \| \varphi \|_{\partial \Omega}.
$$
\end{remark}

\section{Main result}
\label{MR}
In this section, $H$ will denote a complex Hilbert space, $A$ a bounded operator on $H$ and 
$\Omega$ a bounded convex domain of $\C$ with smooth boundary. As we commented in the Introduction,  this smoothness assumption, convenient to avoid technical difficulties in the proofs, may be easily relaxed afterwards.

\begin{theorem}
\label{improved} 
The following uniform bounds holds:
$C_\Omega\leq 1+\sqrt2$.
\end{theorem}
\begin{proof} It suffices to look at the case $C_\Omega>1$. Then, since $C_\Omega$ is the best constant, given $0<\varepsilon<C_\Omega-1$, there exist $f\in {\cal A}(\Omega)$, with $\|f\|_{\infty}=1$, a Hilbert space $H$ and  an operator $A \in B(H)$ with
$\overline W(A)\subset \Omega$ and such that $\lambda =\|f(A)\|\geq C_\Omega{-}\varepsilon >1$.

Let us set $g=C(\overline{f},\cdot) \in \mathcal{A}(\Omega)$ and $S=C(f,A)+C(g,A)^* \in B(H)$. We note that
for the multiplication $fg \in \mathcal{A}(\Omega)$  there holds $fg(A)=g(A)f(A)$, so that
$$
\lambda^2-f(A)^*f(A)=\lambda^2-S^* f(A)+fg (A) .
$$
Moreover, by Lemma~\ref{average}, we have $\| fg \|_\infty \le 1$, and since
\begin{equation}
\label{estimh}
| \lambda^2 + fg | \ge \lambda^2-1 >0,
\end{equation}
we deduce that the mapping $\lambda^2 + fg \in \mathcal{A}(\Omega)$ never vanishes. Therefore, the operator $\lambda^2+fg(A)$ is invertible and we can write
$$
\lambda^2-f(A)^*f(A) = (I-S^*h(A)) (\lambda^2+fg(A)),
$$
where $h=f/(\lambda^2{+}fg) \in  \mathcal{A}(\Omega)$. Next, we observe that the operator $\lambda^2-f(A)^*f(A)$ is singular, whence the factor $ (I-S^*h(A))$ is also singular. Therefore, $1 \le \| S^* h(A) \|$ and then, in view of Lemma~\ref{lemaS} and (\ref{estimh}), we obtain
$$
1 \le 2  \| h(A) \| \le 2\,C_\Omega \| h \|_\infty \le  \frac{2\, C_\Omega}{\lambda^2-1}.
$$
This shows that $(C_\Omega{-}\varepsilon )^2 = \lambda^2 \le 2\,C_\Omega{+}1$, which,  by letting $\varepsilon \to 0_+$, readily yields 
$$
C_\Omega^2 \le 2\, C_\Omega +1, \quad  \text{and thus}  \quad C_\Omega \le 1{+}\sqrt{2}.
$$
\end{proof}

As it is easily checked, Theorem~\ref{improved} remains valid in the complete version. Therefore, since
$$
\frac{1}{2} \left( {1+\sqrt{2} + \frac{1}{1+\sqrt{2}} } \right) = \sqrt{2},
$$
application of Theorem~\ref{improved} and \cite[Theorem~3.1]{dpw} readily leads to the next

\begin{corollary}
\label{cor}
We assume that $W(A) \subset \overline{\Omega}$. Then there holds
$$
w(f(A)) \le \sqrt{2}\  \| f \|_\infty,  \qquad \forall f \in \mathcal{A}(\Omega).
$$ 
\end{corollary}

\begin{remark}
\label{converse}
In principle, estimating $\|f(A) \|$ from the Corollary, in a direct way, would give $\| f(A) \| \le 2 w(f(A)) \le 2\sqrt{2} \|f\|_\infty$. However, the Corollary holds in its complete version indeed and then \cite[Theorem~3.1]{dpw}  shows that the statements in Theorem~\ref{improved} and in Corollary~\ref{cor} are equivalents.
\end{remark}

\begin{remark}
\label{drudry}
In the particular case of the unit disk $\Omega=\D$, Drury \cite{drury} has obtained a more accurate estimate $w(f(A))\leq \frac54\| f \|_\infty$ (see also \cite{klamr}); this constant $\frac54$ is optimal. Previously, it was known that, if furthermore $f(0)=0$, then $w(f(A))\leq \| f \|_\infty$; this result was obtained indepently by Kato \cite{kato} and by Berger and Stampfli \cite{bs}.
\end{remark}

\begin{remark}
\label{final}
For  $\Sigma \subset \C$, set
$$
\mathcal{A}(\Omega,\Sigma) = \{ \, f \in \mathcal{A}(\Omega) \, : \, f(\Omega) \subset \Sigma \, \}.
$$

With very few changes, the proof of Theorem~\ref{improved} also shows that 
$$
\| f(A) \| \le 2 \| f\|_\infty, \qquad \forall f \in \mathcal{A}(\Omega,\Sigma),
$$
whenever $\Sigma \subset \C$ is a sector with vertex at the origin and angle $\pi/2$. To see this, we first fix $A$ and set
$$
C_{\Omega,\Sigma} = \sup \{ \, \| f(A) \| \, : \, f \in \mathcal{A}(\Omega,\Sigma), \, \, \|f\|_\infty =1 \, \} < C_\Omega,
$$
so that we must prove that $C_{\Omega,\Sigma} \le 2$. It suffices to consider the case $1 <  C_{\Omega,\Sigma}$ and we also note that there is no loss of generality in assuming that $\Sigma$ is the sector $|\arg(z)| \le \pi/4$.

Given $\epsilon >0$, we can select $ f \in \mathcal{A}(\Omega,\Sigma)$ such that $\|f\|_\infty =1$ and
$$
\lambda = \|f(A) \| \ge C_{\Omega,\Sigma}-\epsilon >1.
$$
Following the same steps and notations than in the proof of Theorem~\ref{improved}, we  first observe that, in view of Lemma~\ref{lemaS}, the mapping $g=C(\overline{f},\cdot)$  also takes values in $\Sigma$. Therefore, at points $z \in \Omega$ where $f(z) \ne 0$,  we have
$$
h(z)=\frac{f(z)}{\lambda^2{+}f(z)g(z)}=\frac{|f(z)|^2}{\lambda^2 \overline{f(z)}+|f(z)|^2g(z)}
$$
and, since  $\overline{f(z)}+|f(z)|^2g(z) \in \Sigma$, we deduce that $h(z) \in \Sigma$. We thus conclude that  $h \in \mathcal{A}(\Omega,\Sigma)$.  Furthermore, it is also clear that $\Re(fg) \ge 0$, so that  $|\lambda^2 + fg| \ge \lambda^2$, which gives  $\| h \|_\infty \le 1/\lambda^2$. 

Finally,  since $h \in \mathcal{A}(\Omega,\Sigma)$, we have  $ \| h(A) \| \le C_{\Omega,\Sigma} \| h \|_\infty \le C_{\Omega,\Sigma}/\lambda^2$ and, recalling that $1 \le 2 \| h(A) \|$, we obtain
$$
1 \le 2 \| h(A) \| \le C_{\Omega,\Sigma}/\lambda^2 \,  \Rightarrow \, \lambda^2 \le 2 C_{\Omega,\Sigma},
$$
whence
$$
( C_{\Omega,\Sigma}- \epsilon)^2 \le \lambda^2 \le 2 C_{\Omega,\Sigma}.
$$
\end{remark}  

\medskip
\noindent
{ACKNOWLEDGEMENTS: Second author has been financed by the Spanish Ministerio de Econom\'{\i}a y Competitividad under project MTM2014-54710-P.}

\end{document}